\documentclass[10pt, article]{amsart}
\usepackage{tikz}
\usetikzlibrary{calc}
\usepackage{ae} 
\usepackage[T1]{fontenc}
\usepackage[cp1250]{inputenc}
\usepackage{amsmath}
\usepackage{amssymb, amsfonts,amscd,verbatim}

\usepackage[normalem]{ulem}
\usepackage{hyperref}
\usepackage{indentfirst}
\usepackage{latexsym}
\input xy
\xyoption{all}

\usepackage{xcolor}

\usepackage{amsmath}    

\theoremstyle{plain}
\newtheorem{Pocz}{Poczatek}[section]
\newtheorem{Proposition}[Pocz]{Proposition}

\newtheorem{Theorem}[Pocz]{Theorem}
\newtheorem{Corollary}[Pocz]{Corollary}

\newtheorem{Observation}[Pocz]{Observation}

\newtheorem{Problem}[Pocz]{Problem}
\newtheorem{Example}[Pocz]{Example}

\theoremstyle{definition}
\newtheorem{Definition}[Pocz]{Definition}

\theoremstyle{remark}
\newtheorem{Remark}[Pocz]{Remark}

\DeclareMathOperator*{\diam}{diam}

\numberwithin{equation}{section}
\title[Coarse structure of ultrametric spaces with applications]
{Coarse structure of ultrametric spaces with applications}

\author{Yuankui Ma}
\address{Xi'an Technological University, No.2 Xuefu zhong lu, Weiyang district, Xi'an, China 710021}
\email{mayuankui@xatu.edu.cn}

\author{Jeremy Siegert}
\address{Ben-Gurion University of the Negev, Beer-Sheva, Israel}
\email{siegertj@post.bgu.ac.il}

\author{Jerzy Dydak}
\address{University of Tennessee, Knoxville, TN 37996, USA}
\email{jdydak@utk.edu}
\address{Xi'an Technological University, No.2 Xuefu zhong lu, Weiyang district, Xi'an, China 710021}
\email{jdydak@gmail.com}

\date{ \today
}


\begin{document}
\maketitle
\begin{center}
\today
\end{center}

\tableofcontents

\begin{abstract}
We show how to decompose all separable ultrametric spaces into a``Lego" combinations of scaled versions of full simplices. To do this we introduce metric resolutions of large scale metric spaces, which describe how a space can be broken up into roughly independent pieces. We use these metric resolutions to define the coarse disjoint union of large scale metric spaces, which provides a way of attaching large scale metric spaces to each other in a "coarsely independent way". We use these notions to construct universal spaces in the categories of separable and proper metric spaces of asymptotic dimension $0$, respectively. In doing so we generalize a similar result of Dranishnikov and Zarichnyi as well as Nag\'orko and Bell. However, the new application is a universal space for proper metric spaces of asymptotic dimension $0$, something that eluded those authors. We finish with a description of some countable groups that can serve as such universal spaces.

\end{abstract}

\section{Introduction}

Coarse geometry considers metric spaces and abstracts spaces from a ``large scale" perspective, wherein one pays most attention to the properties that persist through ``zooming out" from the object of study while ignoring whatever small scale structure that may exist. Such ideas have existed since the early $20^{th}$ century but were first explicitly layed out by Gromov in \cite{Grom} where several properties relevant to the coarse geometry of countable groups were clearly defined. Among these was the asymptotic dimension of a metric space, which was intended to serve as a large scale analog of the classical covering dimension. This particular property would come to some prominence when Yu proved in \cite{Yu} that the Novikov conjecture is satisfied by groups of finite asymptotic dimension. Such progress and the intuitive appeal of the subject has motivated much of the study into developing the methods of coarse geometry.
\vspace{\baselineskip}

In this paper we completely characterize the coarse geometry of a large class of spaces. Specifically separable ultrametric spaces. The breadth of this class of spaces was realized when Brodskiy, Dydak, Higes, and Mitra proved in \cite{BDHM} that every separable metric space of asymptotic dimension $0$ is coarsely equivalent to a countable ultrametric space whose metric take values in the nonnegative integers. Such spaces include all countable locally finite groups (i.e. those where every finitely generated subgroup is finite). We characterize these spaces by showing how to decompose them into several scaled full simplices on certain sets of vertices. This in itself is via two concepts. The first of which is metric resolution (introduced in Section \ref{MetricResolutions}) which provides a way of distinguishing particular ``pieces" of a metric space. The second concept is that of a coarse disjoint union (introduced in Section \ref{CoarseDisjointUnions}), which describes how to put together families of metric spaces together into one space such that each metric space used in the construction is ``coarsely independent" from the other spaces used in the construction. Our characterization of ultrametric spaces is then performed by distinguishing certain pieces, each of which is a scaled full simplex, of the spaces via metric resolutions, then attaching them together via a coarse disjoint union in such a way that the resulting space is coarsely equivalent to the original. In doing this we find that there are only countably many such scaled full simplices that can be employed which alloys us in Section \ref{universal spaces} to construct universal spaces for the categories of separable and proper metric spaces of asymptotic dimension $0$, respectively. In the case of separable spaces of bounded geometry such a universal space has been constructed previously by Dranishnikov and Zarichnyi in \cite{DZ}. In the broader case of separable spaces, this has been done previously by Nagorko and Bell in \cite{BN}. T. Banakh and I. Zarichnyy constructed universal spaces for coarsely homogeneous spaces of asymptotic dimension $0$ in \cite{BZ}. The methods employed in these papers differ from the methods we employ. The universal space for proper metric spaces of asymptotic dimension $0$ is the first that appears in the literature. The question of whether or not there are universal spaces for proper metric spaces of asymptotic dimension greater than $0$ is open and seems to be an interesting question. We finish with some discussion of infinite groups that can be realized as such universal spaces.

The authors are grateful to Henryk Toru\'nczyk for helpful comments about the paper.\\
The authors are extremely grateful to the excellent referee for pointing out errors and gaps in the original version of the paper.\\
Jeremy Siegert was supported by the Israel Science Foundation grant No. 2196/20.

\section{Preliminaries}\label{Preliminaries}

We begin with the few basic preliminary definitions needed in subsequent sections. 

\begin{Definition}
A metric space $(X,d)$ is called an \textbf{ultrametric space} if in place of the usual triangle inequality for metric spaces the metric $d$ satisfies the stronger \emph{ultrametric triangle inequality} which says that for all $x,y,z\in X$

\[d(x,z)\leq\max\{d(x,y),d(y,z)\}\]
\end{Definition}

\begin{Definition}
Given a set $D$ of non-negative integers, a \textbf{$D$-ultrametric space} is an ultrametric space with all distances belonging to $D$.
If $D$ is the set of all non-negative integers, then a $D$-ultrametric space will be called an \textbf{integral ultrametric space}.
\end{Definition}

\begin{Definition}
A metric space $(X,d)$ is said to be of \text{asymptotic dimension} $0$ if for every uniformly bounded cover $\mathcal{U}$ of $X$, there is a uniformly bounded cover $\mathcal{V}$ that is refined by $\mathcal{U}$ and whose elements are disjoint. 
\end{Definition}

Alternatively, one could define for each $r>0$ the relation $\sim_{r}$ on $X$ be setting $x\sim_{r}y$ if $d(x,y)<r$. Then say that $x,y\in X$ are \textbf{r-connected} if there is a finite chain of elements $x=y_{0},y_{1},\ldots,y_{n}=y$ such that $y_{i}\sim_{r}y_{i+1}$ for $0\leq i\leq n-1$ and define the $\textbf{r-components}$ of $X$ to be maximally $r$-connected subsets of $X$. An metric space $(X,d)$ is of asymptotic dimension $0$ if and only if the collection of $r$-components of $X$ is uniformly bounded for every $r>0$. For a more in depth discussion of asymptotic dimension the reader is referred to \cite{Roe lectures}.

\begin{Definition}
A function $f:(X,d_{1})\rightarrow(Y,d_{2})$ of metric spaces is called:
\begin{enumerate}

\item \textbf{uniformly bornologous} (or \textbf{large scale continuous}) if for all $R>0$ there is an $S>0$ such that if $d_{1}(x,y)\leq R$ then $d_{2}(f(x),f(y))\leq S$.
\item \textbf{proper} if for every bounded $B\subseteq Y$, $f^{-1}(B)$ is bounded in $X$.
\item \textbf{uniformly proper} if for every $R>0$ there is an $S>0$ such that if $B\subseteq Y$ is bounded by $R$, then $f^{-1}(B)$ is bounded by $S$.
\item \textbf{coarsely surjective} if there is an $R>0$ such that for every $y\in Y$ there is an $x\in X$ such that $d_{2}(f(x),y)\leq R$.
\item a \textbf{coarse equivalence} if it is uniformly bornologous, uniformly proper, and coarsely surjective.
\item a \textbf{coarse embedding} if it is uniformly bornologous and uniformly proper.
\end{enumerate}
\end{Definition}

The following construction provides a universal construction of ultrametrics for functions $f:X\to (Y,d_Y)$ from sets to metric spaces:

\begin{Definition}\label{UniversalUltrametricDef}
Given a function $f:X\to (Y,d)$ from a set to a metric space we define
the \textbf{induced integral ultrametric} $d_f^{ul}(x,y)$ by $f$ and $d$ as the  minimum of all integers $r \ge 1$ such that there is an $r$-chain in $Y$ joining $f(x)$ and $f(y)$.
$d_f^{ul}(x,y)=0$ if $x=y$.
\end{Definition}

Recall that an \textbf{$r$-chain} joining $a,b\in Y$ is a sequence $x_0=a,\ldots,x_k=b$
such that $d(x_i,x_{i+1}) < r$ for each $i < k$.

\begin{Proposition}
If $f:(X,d_X)\to (Y,d)$ is surjective large scale continuous and $(Y,d)$ is of asymptotic dimension $0$, then $f:(X,d_f^{ul})\to (Y,d)$ is a coarse equivalence.
\end{Proposition}
\begin{proof}
Choose a selection $s:Y\to X$ for $f$.
Since $id_d^{ul}(x,y)\leq d(f(x),f(y))+1$ for all $x,y\in X$, the function $s:(Y,d)\to (X,d_f^{ul})$ is large scale continuous. If $(Y,d)$ is of asymptotic dimension $0$ and $r > 0$
there is $S > 0$ such that every two points in $Y$ that can be connected by an $r$-chain have to be of distance at most $S$. That shows
$f:(X,d_f^{ul})\to (Y,d)$ is a large scale continuous map as well.

\end{proof}

\begin{Corollary}\label{Asdim0ViaUltrametrics}
Suppose $(X,d)$ is a metric space. The following conditions are equivalent:\\
1. $(X,d)$ is of asymptotic dimension $0$.\\
2. $id:(X,d_{id}^{ul})\to (X,d)$ is large scale continuous.\\
3. $id:(X,d)\to (X,d_{id}^{ul})$ is a coarse equivalence.
\end{Corollary}

The following is a generalization of ultrametric spaces.

\begin{Definition}\label{IsoscelesSpaceDef}
A metric space $(X,d)$ is called an \textbf{isosceles space} if every triangle in $X$ is isosceles.
\end{Definition}

\begin{Proposition}
Every isosceles space $(X,d)$ is of asymptotic dimension at most $0$.
\end{Proposition}
\begin{proof}
In view of \ref{Asdim0ViaUltrametrics} it suffices to show $id:(X,d_{id}^{ul})\to (X,d)$ is large scale continuous. Suppose $x_0,\ldots,x_k$ is an $m$-chain joining $x$ to $y$ in $X$. If $k\leq 2$, then obviously $d(x,y) < 2m$.
If $k > 2$, then $d(x_0,x_3) < 2m$ and continuing by induction, we conclude
$d(x,y) < 2m$. Thus, $id:(X,d_{id}^{ul})\to (X,d)$ is $2$-Lipschitz, hence large scale continuous.
\end{proof}

The following result from \cite{BDHM} is the critical observation needed to construct the universal spaces in section \ref{universal spaces}. Our version is slightly more general and it follows from \ref{Asdim0ViaUltrametrics}.

\begin{Theorem}\label{countableintegral}
If $(X,d_X)$ is a separable metric space of asymptotic dimension zero, then
there is a countable integral ultrametric space $(Y,d)$ coarsely equivalent to $(X,d_X)$. Moreover, if $X$ is proper (in the sense that its bounded subsets have compact closure), then $Y$ can be chosen to have finite bounded subsets only.
\end{Theorem}
\begin{proof}
Choose a maximal subset $Y$ of $X$ such that $d_X(x,y) > 1$ if $x\ne y$ in $Y$.
Use the integral ultrametric $d$ on $Y$ from \ref{Asdim0ViaUltrametrics}.
\end{proof}

\section{Metric resolutions}\label{MetricResolutions}

In this section we introduce a formal definition of ``Lego" combinations of metric spaces.
\begin{Definition}\label{MetricResolutionDef}
A \textbf{metric resolution} is a surjective function $f:X\to S$ of metric spaces so that $d_X(x,y)=d_S(f(x),f(y))$ if $f(x)\ne f(y)$.
\end{Definition}

\begin{Example}
$d_f^{ul}$
from Definition \ref{UniversalUltrametricDef}, provided $f:X\to Y$ is surjective,
is a metric resolution if $Y$ is given the integral ultrametric induced by $d_Y$.
\end{Example}

In applications, metric resolutions will be constructed as indexing functions $i:\coprod\limits_{t\in S} X_s\to S$ (that means $i(x)=t$ if $x\in X_t$)
from a disjoint union $ \coprod\limits_{t\in S} X_s$ of metric spaces $(X_t,d_t)$ so that there is a metric $d_S$ on $S$ and the metric $d$ on $ \coprod\limits_{t\in S} X_t$ having the following properties:\\
1. $d$ restricted to each $X_t$ equals $d_t$,\\
2. $d(x,y)=d_S(u,t)$ if $x\in X_u$, $y\in X_t$, and $u\ne t$.

\begin{Proposition}\label{ExistenceOfMetricResolutions}
Suppose $f:X\to S$ is a surjective function and $d_S$ is a metric on $S$. Suppose $d_t$ is a metric on $X_t:=f^{-1}(t)$ for each $t\in S$.
If the diameter of each $X_t$ is at most $0.5\cdot d_S(u,t)$ for all $u\ne t$, then there is a unique metric $d$ on 
$X$ making $f$ a metric resolution. Moreover,
if $(X_t,d_t)$, $t\in S$, are isosceles and $(S,d_S)$ is isosceles, then the total $(X,d)$ is isosceles as well.
\end{Proposition}
\begin{proof}
Obviously, there is at most one $d$ that can be considered as the total metric on $X$. Given $x,y,z\in X$, we need to show the Triangle Inequalities in each of the two cases. It is so if they belong to three different fibers of $f$. Also, if $x, z$ belong to the same fiber of $f$.
Assume $x\in X_a$, $z\in X_b$, and $a\ne b$. The Triangle Inequalities are obvious if $y\in X_a$ or $y\in X_b$ due to the requirement that the diameter of each $X_t$ is at most $0.5\cdot d_S(u,t)$ for all $u\ne t$.
\end{proof}

\begin{Proposition}\label{ExistenceOfUltraMetricResolutions}
Suppose $f:X\to S$ is a surjective metric resolution and $d_S$ is an ultrametric on $S$.
If the diameter of each $X_t:=f^{-1}(t)$ is at most $d_S(u,t)$ for all $u\ne t$, $(X_t,d|X_t)$, $t\in S$, are ultrametrics, then the total $(X,d)$ is ultrametric as well.
\end{Proposition}
\begin{proof}
Given $x,y,z\in X$, we need to show the Ultrametric Triangle Inequality. It is so if they belong to three different fibers of $f$. Also, if $x, z$ belong to the same fiber of $f$.
Assume $x\in X_a$, $z\in X_b$, and $a\ne b$. The Ultrametric Triangle Inequality are obvious if $y\in X_a$ or $y\in X_b$ due to the requirement that the diameter of each $X_t$ is at most $d_S(u,t)$ for all $u\ne t$. 
\end{proof}

\begin{Proposition}\label{UltrametricResolution}
Suppose $(X,d)$ is an ultrametric space, $x_0\in X$, and $S$ is the set of all possible values of $d(x_0,x)$, $x\in X$. If $d_S$ is defined by $d(u,t)=\max(u,t)$ for $u\neq t$, then $d_S$ is an ultrametric on $S$ and $i:X\to S$ defined by $i(x)=d(x_0,x)$ is a metric resolution.
\end{Proposition}
\begin{proof}
$d_S$ is clearly an ultrametric. Suppose $i(x)=u < t=i(y)$. By the Ultrametric Triangle Inequality
we get $d(x,y)=t=\max(u,t)=d_S(i(x),i(y))$.
\end{proof}

\begin{Proposition}\label{FiniteUltrametricResolution}
Suppose $D$ is a finite subset of non-negative integers and $(X,d)$ is a $D$-ultrametric space. If $m=\max(D)$, then there is a metric resolution
$i:X\to m\cdot C$ for some full simplex $C$.
\end{Proposition}
\begin{proof}
Consider the equivalence relation $\sim$ on $X$ defined by $x\sim y$ if $d(x,y)< m$. Let $C$ be the set of all equivalence classes. Notice that
$d(x,y)=m$ if $x$ and $y$ belong to different equivalence classes.
Define $i(x)$ as the equivalence class of $[x]$ of $x$.
\end{proof}

\begin{Proposition}
Suppose $(X,d)$ is an ultrametric space. The function $\rho(A,B)=\sup\{d(x,y)| x\in A, y\in B\}$ for $A\ne B$ defines an ultrametric on the space of non-empty bounded subsets of $X$.
\end{Proposition}
\begin{proof}
If $\rho(A,C) > \max(\rho(A,B),\rho(B,C))$, then there exist points
$a\in A$ and $c\in C$ such that $d(a,c) > \max(\rho(A,B),\rho(B,C))$.
Choose $b\in B$. Now, $d(a,c)\leq \max(d(a,b),d(b,c))\leq \max(\rho(A,B),\rho(B,C))$, a contradiction.
\end{proof}

\begin{Proposition}
If $i:X\to S$ is a metric resolution and $X$ is an ultrametric space,
then the distance between $u$ and $t$ in $S$, $u\ne t$, is the same as the distance between $i^{-1}(u)$ and $i^{-1}(t)$ in $X$.
\end{Proposition}
\begin{proof}
If $x\in i^{-1}(u)$ and $y\in i^{-1}(t)$, then $d_X(x,y)=d_S(u,t)$.
Therefore the distance between $u$ and $t$ in $S$, $u\ne t$, is the same as the distance between $i^{-1}(u)$ and $i^{-1}(t)$ in $X$.
\end{proof}

\section{Splicing of metrics}

In this section we will discuss a more general concept than both $d_f^{ul}$
in Definition \ref{UniversalUltrametricDef} and metric resolutions in Section \ref{MetricResolutions}. The idea is to splice metrics on the fibers of a function $f:X\to S$ with a metric $d_S$ on $S$. The construction in this section is on the other side of spectrum than metric resolutions. There, fibers are basically parallel, in the new concept different fibers diverge at infinity.

\begin{Definition}\label{SplicingMetricsViaASection}
Given a surjective function $f:X\to S$ from a set to a metric space $(S,d_S)$, given metrics $d_t$ on fibers $f^{-1}(t)$, $t\in S$, and given a section $g:S\to X$, we define the \textbf{splicing $d$ of metrics along $g$} on $X$ as follows:\\
1. $d$ restricted to each fiber $f^{-1}(t)$ equals $d_t$.\\
2. If $x\in f^{-1}(u)$, $y\in f^{-1}(t)$ and $u\ne t$, then
$$d(x,y)=\max(d_S(u,t),d_u(x,g(u)),d_t(y,g(t))).$$
\end{Definition}

\begin{Example}
A metric resolution $f:X\to S$ is a splicing of metrics on fibers with a metric on $S$ along any section $g:S\to X$, i.e. all sections give the same metrics, if the diameters of fibers satisfy the following condition: $\diam(f^{-1}(u))\leq dist(u,S\setminus \{u\})$ for all $u\in S$.
\end{Example}
\begin{proof}
In this case the formula $d(x,y)=\max(d_S(u,t),d_u(x,g(u)),d_t(y,g(t)))$
yields $d_S(u,t)$ if $u\ne t$.
\end{proof}

\begin{Proposition}
A splicing $d$ of metrics is a metric. If the metric $d_S$ on $S$ and all the metrics on the fibers of $f$ are ultrametrics, then so is $d$.
\end{Proposition}
\begin{proof}
Suppose $d(x,z) > d(x,y)+d(y,z)$. Obviously, all three points $f(x), f(y), f(z)$ cannot be equal. \\
Case 1: $u=f(x)=f(z)\ne t=f(y)$. Notice $d(x,y)\ge d_u(x,g(u))$ and
$d(z,y)\ge d_u(z,g(u))$, so $d(x,y)+d(y,z)\ge d_u(x,z)$, a contradiction.\\
Case 2: $u=f(x)\ne t=f(z)$ and $v=f(y)$. \\
If $d(x,z)=d_{S}(u,t)$ and $v=u$ then $d(y,x)=d_{u}(x,y)$ and\\
 $d(y,z)=\max(d_{S}(u,t),d_{u}(y,g(u)),d_{t}(z,g(t)))$. Then our assumption that $d(x,z)>d(y,x)+d(y,z)$ becomes $d_{S}(u,t)>d_{u}(x,y)+\max(d_{S}(u,t),d_{u}(y,g(u)),d_{t}(z,g(t)))\geq d_{S}(u,t)$. That is, $d_{S}(u,t)>d_{S}(u,t)$, a contradiction. Therefore we must have that $v\neq u$. A similar contradiction appears if we assume that $v=t$. We must then have that $v$ is equal to neither $u$ nor $t$. Now, with our original assumption that $d(x,z)>d(y,x)+d(y,z)$, we have that $d(x,y)+d(y,z)\ge d_S(u,v)+d_S(v,t)\ge d_S(u,t)$, a contradiction. Without loss of generality, assume $d(x,z)=d_u(x,g(u)) > d_S(u,t)$. Now, $v$ must be equal to $u$, so 
$d(x,y)+d(y,z)=d_u(x,y)+\max(d_u(y,g(u)),d_S(u,t),d_t(g(t),z))\ge 
\max(d_u(x,y)+d_u(y,g(u)),d_S(u,t),d_t(g(t),z))\ge
\max(d_u(x,g(u)),d_S(u,t),d_t(g(t),z))=d(x,z)$, a contradiction again.

Assume the metric $d_S$ on $S$ and all the metrics on the fibers of $f$ are ultrametrics.
Suppose $d(x,z) > \max(d(x,y),d(y,z))$. Obviously, all three points $f(x), f(y), f(z)$ cannot be equal. \\
Case A: $u=f(x)=f(z)\ne t=f(y)$. Notice $d(x,y)\ge d_u(x,g(u))$ and
$d(z,y)\ge d_u(z,g(u))$, so $\max(d(x,y),d(y,z))\ge d_u(x,z)$, a contradiction.\\
Case B: $u=f(x)\ne t=f(z)$ and $v=f(y)$. 
If $d(x,z)=d_S(u,t)$, then $v\ne u,t$ (as in the Case 2 above) and $\max(d(x,y),d(y,z))\ge \max(d_S(u,v),d_S(v,t))\ge d_S(u,t)$, a contradiction. Without loss of generality, assume $d(x,z)=d_u(x,g(u))$. Now, $v$ must be equal to $u$, so \\
$\max(d(x,y),d(y,z))=\max(d_u(x,y),\max(d_u(y,g(u)),d_S(u,t),d_t(g(t),z)))= $\\
$\max(d_u(x,y),d_u(y,g(u)),d_S(u,t),d_t(g(t),z))\ge \max(d_u(x,g(u)),d_S(u,t),d_t(g(t),z))=d(x,z)$, a contradiction again.
\end{proof}

\section{Coarse disjoint unions}\label{CoarseDisjointUnions}
In this section we define the concept of a coarse disjoint union, which is a means of joining a family of metric spaces together in such a way that they are "coarsely independent" of one another. 

\begin{Definition}\label{DefCoarseDisjointUnion}
Given a family $\{(X_s,d_s)\}_{s\in S}$ of metric spaces, a \textbf{coarse disjoint union} of that family is a disjoint union $\coprod\limits_{s\in S} X_s$
equipped with a metric $d$ satisfying the following properties:\\
1. $d$ restricted to each $X_s$ equals $d_s$.\\
2. Given $M > 0$ there are bounded subsets $B_s$ of $X_s$, all but finitely many of them empty, such that if $x\in X_s\setminus B_s$ and $y\in X_t\setminus B_t$ for some $s\ne t$, then $d(x,y) > M$.
\end{Definition}

\begin{Observation} Notice that Condition 2 in \ref{DefCoarseDisjointUnion} can be split into the following two conditions:\\
2a. Every bounded subset $B$ of $\coprod\limits_{s\in S} X_s$ is contained
in $\coprod\limits_{s\in F} X_s$ for some finite $F\subset S$.\\
2b. Given $M > 0$ there is a bounded subset $B$ of $\coprod\limits_{s\in S} X_s$ such that if $x\in X_s\setminus B$ and $y\in X_t\setminus B$ for some $s\ne t$, then $d(x,y) > M$.
\end{Observation}

\begin{Observation} 
If a coarse disjoint union $X$ exists and each $X_s$ is non-empty, then $S$ is countable. Indeed, $X$ is the union of its balls $B(x_0,n)$, $n\ge 1$, and each ball intersects finitely many terms $X_s$.
\end{Observation}

\begin{Proposition}
A disjoint union $\coprod\limits_{s\in S} X_s$ is a coarse disjoint union of a family $\{(X_s,d_s)\}_{s\in S}$ of metric spaces if and only if it can be
equipped with a metric $d$ satisfying the following properties:\\
1. $d$ restricted to each $X_s$ equals $d_s$.\\
2. Given a sequence $\{x_n\}_{n\ge 1}$ of points in $\coprod\limits_{s\in S} X_s$ belonging to different parts $X_s$, one has $x_n\to\infty$ (that means $d(a,x_n)\to\infty$ for some, hence for all, $a\in \coprod\limits_{s\in S} X_s$).\\
3. Given $M > 0$ and a sequence of pairs $(x_n,y_n)$, $n\ge 1$, of points in $\coprod\limits_{s\in S} X_s$ such that $d(x_n,y_n) < M$ for all $n\ge 1$ and $x_n\to\infty$, there is $k\ge 1$ such that for each $n\ge k$ there is an index $s\in S$ so that $x_n,y_n\in X_s$.
\end{Proposition}
\begin{proof}
Suppose $\coprod\limits_{s\in S} X_s$ is a coarse disjoint union in the sense of Definition \ref{DefCoarseDisjointUnion}. Given a sequence $\{x_n\}_{n\ge 1}$ of points in $\coprod\limits_{s\in S} X_s$ belonging to different parts $X_s$ such that $d(a,x_n)$ is not divergent to infinity for some $a\in \coprod\limits_{s\in S} X_s$, we may reduce this case to the one where there is $M > 0$ satisfying
$d(a,x_n) < M$ for all $n\ge 1$. There are bounded subsets $B_s$ of $X_s$, all but finitely many of them empty, such that if $x\in X_s\setminus B_s$ and $y\in X_t\setminus B_t$ for some $s\ne t$, then $d(x,y) > 2M$.
There are $t\ne s$ in $S$ such that $B_t=B_s=\emptyset$ and $x_k\in X_t$, $x_m\in X_s$ for some $k,m$, a contradiction as $d(x_k,x_m) < 2M$.

Given $M > 0$ and a sequence of pairs $(x_n,y_n)$, $n\ge 1$, of points in $\coprod\limits_{s\in S} X_s$ such that $d(x_n,y_n) < M$ for all $n\ge 1$ and $x_n\to\infty$, assume there is no $k\ge 1$ such that for each $n\ge k$ there is an index $s\in S$ so that $x_n,y_n\in X_s$. We may reduce this case to the one where $x_n$ and $y_n$ do not belong to the same $X_s$ for all $n\ge 1$. There are bounded subsets $B_s$ of $X_s$, all but finitely many of them empty, such that if $x\in X_s\setminus B_s$ and $y\in X_t\setminus B_t$ for some $s\ne t$, then $d(x,y) > M$.
There are $t\ne s$ in $S$ such that $B_t=B_s=\emptyset$ and $x_k\in X_t$, $y_k\in X_s$ for some $k$, a contradiction as $d(x_k,y_k) < M$.

The proof in the reverse direction is similar.
\end{proof}

\begin{Corollary}\label{CoarseDisjointUnionOfPoints}
A metric space $(S,d_S)$ is a coarse disjoint union of its points if and only if it is proper (its bounded subsets are finite) and has the property
that $d_S(x_n,y_n) < M <\infty$ for all $n$ and $d_S(x_n,x_1)\to\infty$ imply existence of $k$ such that $x_n=y_n$ for all $n > k$.
\end{Corollary}

\begin{Proposition}\label{MetricResolutionsVsCoarseDisjointUnions}
Suppose $f:X\to S$ is a surjective function and $X$ is given the metric equal to splicing of metrics on fibers along a section $g:S\to X$. If $S$ is a coarse disjoint union of its points,
then $X$ is a coarse disjoint union of fibers of $f$.
\end{Proposition}
\begin{proof}
Apply \ref{CoarseDisjointUnionOfPoints}:\\
If $B$ is a bounded subset of $X$, then $f(B)$ is bounded as $f$ is $1$-Lipschitz. Hence $f(B)$ is finite.

If $d(x_n,y_n) < M$, $x_n\to \infty$, and $f(x_n)\ne f(y_n)$ for all $n\ge 1$, then $d_S(f(x_n),f(y_n)) < M$ and $f(x_n)\to\infty$. Hence $f(x_n)=f(y_n)$ for almost all $n$, a contradiction.
\end{proof}

\begin{Corollary}\label{UltrametricAsCoarseDisjointUnion}
Every integral ultrametric space $(X,d)$ is a coarse disjoint union of countably many of its bounded subsets.
\end{Corollary}
\begin{proof}
Pick $x_0\in X$ and let $S$ be the set of all possible distances $d(x_0,x)$, $x\in X$. The metric $d_S$ of $S$ is defined by $d_S(u,t)=\max(u,t)$ if $u\ne t$. By \ref{UltrametricResolution} the function $f:X\to S$ given by $f(x)=d(x_0,x)$ is a metric resolution. Notice $S$ is a coarse disjoint union of its points and apply \ref{MetricResolutionsVsCoarseDisjointUnions}.
\end{proof}

\begin{Corollary}
If $S$ is countable, then any family $\{(X_s,d_s)\}_{s\in S}$ of non-empty metric spaces has a coarse disjoint union. Moreover, if each $d_s$ is an (integral) ultrametric, then there is a coarse disjoint union equipped with an (integral) ultrametric.
\end{Corollary}
\begin{proof}
We may assume $S=\mathbb{N}$ and we equip $S$ with the $\max$ ultrametric $d(m,n)=\max(m,n)$ if $m\ne n$. Let $X$ be a disjoint union of all $X_s$, $s\in S$, equipped with a splicing metric for some section $g:S\to X$.
By \ref{CoarseDisjointUnionOfPoints} and \ref{MetricResolutionsVsCoarseDisjointUnions}, $X$ is a coarse disjoint union of $\{(X_s,d_s)\}_{s\in S}$.
\end{proof}

\begin{Proposition}\label{TwoDifferentCoarseUnionsProp}
Given two coarse disjoint unions $(\coprod\limits_{s\in S} X_s,d)$ and $(\coprod\limits_{s\in S} Y_s,\rho)$ \\
1. isometric embeddings $i_s:X_s\to Y_s$, $s\in S$, induce a coarse embedding $i$ from $(\coprod\limits_{s\in S} X_s,d)$ to $(\coprod\limits_{s\in S} Y_s,\rho)$,\\
2. identity functions $i_s:X_s\to Y_s$, $s\in S$, induce a coarse equivalence $i$ from $(\coprod\limits_{s\in S} X_s,d)$ to $(\coprod\limits_{s\in S} Y_s,\rho)$.
\end{Proposition}
\begin{proof} Notice 1) implies 2), so only 1) needs to be proved.
Suppose, on the contrary, that there is $M > 0$ and a sequence of pairs $(x_n,y_n)$, $n\ge 1$, of points in $\coprod\limits_{s\in S} X_s$ such that $d(x_n,y_n) < M$ for all $n\ge 1$ but
$\rho(i(x_n),i(y_n))\to\infty$. There is $k\ge 1$ such that for each $n\ge k$ there is an index $s\in S$ so that $i(x_n),i(y_n)\in Y_s$. Therefore $\rho(i(x_n),i(y_n))=d(x_n,y_n)$ for all $n > k$, a contradiction.
\end{proof}

\begin{Proposition}\label{SeparabilityPropernessOfUnions}
Suppose $S$ is countable and $\{(X_s,d_s)\}_{s\in S}$ is a family of metric spaces. A coarse disjoint union of $\{(X_s,d_s)\}_{s\in S}$ is separable (proper) if and only if each $X_s$ is separable (proper).
\end{Proposition}
\begin{proof}
It follows from the fact any bounded subset $B$ of the coarse disjoint union is a union of finitely many bounded subsets of some among spaces $X_s$.
\end{proof}

\section{Special ultrametric spaces}\label{SpecialUltrametricSpaces}

In this section we construct universal spaces (with respect to isometric embeddings) for specific classes of bounded ultrametric spaces. More specifically, for a finite subset $D\subseteq\mathbb{N}$ that contains $0$ we construct universal spaces for the class of countable $D$-ultrametric spaces, and for each $m\geq 1$ we construct an universal space for the class of $D$-ultrametric spaces with at most $m$ points. The spaces constructed in this section serve as the building blocks for the universal spaces constructed in section \ref{universal spaces}.

\begin{Proposition}\label{IsometricembeddingsAndResolutions}
Suppose $p_i:E_i\to B_i$, $i=1,2$, are metric resolutions, $f:B_1\to B_2$ is an isometric embedding and for each $s\in B_1$ there is an isometric embedding $f_s:p_1^{-1}(s)\to p_2^{-1}(f(s))$. There is an isometric embedding $F:E_1\to E_2$ such that $p_2\circ F=f\circ p_1$.
\end{Proposition}
\begin{proof}
$F$ is the union of all $f_s$, $s\in B_1$. Clearly, it is an isometric embedding.
\end{proof}

The following can be seen as an inductive step in splitting bounded ultrametric spaces into Lego pieces.
\begin{Proposition}\label{BasicSplitOfUltrametrics}
Suppose $D$ is a finite set of non-negative integers and $k=\max(D) > 0$.
Every $D$-ultrametric space $X$ of diameter $k$ admits a surjective metric resolution
$f:X\to \{0,k\}\subset \mathbb{N}$ such that $\diam(f^{-1}(0)) < k$.
\end{Proposition}
\begin{proof}
Pick two points $a,b\in X$ at distance $k$. $f$ maps all points $x\in X$ such that $d(x,a) < k$ to $0$ and $f$ maps all points $x\in X$ such that $d(a,x)=k$ to $k$.
Notice $\diam(f^{-1}(0)) < k$, $\diam(f^{-1}(k)) \leq k$, so $f$ is indeed a metric resolution.
\end{proof}

\begin{Proposition}\label{boundedfiniteuniversal}
Given a finite set $D$ of non-negative integers and given $m\ge 1$ there is a finite $D$-ultrametric space $FU(m,D)$ such that any $D$-ultrametric space $X$ containing at most $m$ points isometrically embeds in $FU(m,D)$.
\end{Proposition}
\begin{proof}
We will construct $FU(m,D)$ by induction on $m$.
Let $FU(1,D)$ be a one-point metric space for any $D$. 
Given $D$ with $k=\max(D) > 0$, $FU(m,D)$ is the disjoint union
of $FU(m-1,D\setminus \{k\})$ and $FU(m-1,D)$ so that any distance between points in different parts is $k$. Thus, as in \ref{BasicSplitOfUltrametrics}, $FU(m,D)$ admits a metric resolution
onto $\{0,k\}$ with fibers $FU(m-1,D\setminus \{k\})$ and $FU(m-1,D)$.

Now, by induction on $m$, we can show the space $FU(m,D)$ has the needed property. Indeed, it is so for $m=1$. Once we assume
the spaces $FU(n,C)$, $n < m$ and $C$ a finite set of non-negative integers, have the needed property, we use
\ref{BasicSplitOfUltrametrics} jointly with \ref{IsometricembeddingsAndResolutions} to conclude $FU(m,D)$ has the property 
that any $D$-ultrametric space $X$ containing at most $m$ points isometrically embeds in $FU(m,D)$.
\end{proof}

\begin{Proposition}\label{boundedcountableuniversal}
Given a finite set $D$ of non-negative integers there is a countable $D$-ultrametric space $CU(D)$ such that any countable $D$-ultrametric space $X$ isometrically embeds in $CU(D)$.
\end{Proposition}
\begin{proof}
Let $CU(\{0\})$ be a one-point metric space. Suppose spaces $CU(D)$ are known for all $D$ containing at most $n$ integers and $G$ contains $(n+1)$ integers with $k=\max(G)$. 
Define $CU(G)$ as the total space in a metric resolution $f:CU(G)\to k\cdot \Delta$, where $k\cdot\Delta$ is an infinite countable space with all non-zero distances equal $k$ (i.e. $k$ times the infinite simplex), so that $f^{-1}(x)$
is a copy of $CU(G\setminus\{k\})$ for each $x\in k\cdot\Delta$.
By \ref{ExistenceOfMetricResolutions} and
\ref{ExistenceOfUltraMetricResolutions} such a resolution exists and $CU(G)$ is an ultrametric space.

Given a countable $G$-ultrametric space $X$ if $\diam(X) < k$, $X$ clearly embeds into $CU(G)$ isometrically due to the inductive step,
so assume $\diam(X)=k$ and choose a maximal subset $S$ of $X$ such that all non-zero distances in $S$ are equal $k$. For each $s\in S$ consider
the set $X_s$ of all $x\in X$ such that $d(s,x) < k$. Notice the following:\\
a. $\diam(X_s) < k$
for each $s\in S$, \\
b. if $x\in X_s$, $y\in X_t$ and $s\ne t$, then $d(x,y)=k$, \\
c. $X=\bigcup\limits_{s\in S}X_s$. \\
Choose an injective function $g:S\to k\cdot\Delta$ and choose isometric embeddings $f_s:X_s\to f^{-1}(g(s))$ for each $s\in S$.
Apply \ref{IsometricembeddingsAndResolutions}.
\end{proof}

\begin{Corollary}\label{UniversalForCountableUltrametric}
There is a countable integral ultrametric space $CU$ such that any countable integral ultrametric space $X$ isometrically embeds in $CU$.
\end{Corollary}
\begin{proof}
Consider $\mathbb{N}$ equipped with the ultrametric $d(k,m)=\max(k,m)$
for $k\ne m$
and use spaces $ CU(\mathbb{N}\cap [0,i])$ from \ref{boundedcountableuniversal}.
Define $CU$ as the total space in a metric resolution $f:CU\to \mathbb{N}$ so that $f^{-1}(i)=CU(\mathbb{N}\cap [0,i])$ for each $i\in \mathbb{N}$.
By \ref{ExistenceOfMetricResolutions} and \ref{ExistenceOfUltraMetricResolutions} such a resolution exists
and by \ref{IsometricembeddingsAndResolutions} together with \ref{UltrametricResolution} the space $CU$ has the needed property.
\end{proof}

\section{Universal spaces}\label{universal spaces}

In this section we prove our main applications. That is, we give a detailed construction of universal spaces (with respect to coarse embeddings) in the classes of separable metric spaces of asymptotic dimension $0$ and the class of proper metric spaces of asymptotic dimension $0$. 

\begin{Theorem}
There is a countable integral ultrametric space $CU$ such that any separable metric space $X$ of asymptotic dimension $0$ coarsely embeds in $CU$.
\end{Theorem}
\begin{proof}
Consider $CU$ from \ref{UniversalForCountableUltrametric}. We claim that $CU$ is the desired universal space. In light of Theorem \ref{countableintegral} it will suffice to show that if $(X,d)$ is a countable integral ultrametric space, then $X$ coarsely embeds into $CU$. Then let $(X,d)$ be such a space. Apply \ref{UniversalForCountableUltrametric}.
\end{proof}

For the following two results, let $D_{1},D_{2},\ldots$ be an enumeration of the finite subsets of $\mathbb{N}$ that contain $0$. 

\begin{Theorem}\label{ProperDimZeroCase}
There is a countable and proper integral ultrametric space $PU$ such that any proper metric space $X$ of asymptotic dimension $0$ coarsely embeds in $PU$.
\end{Theorem}
\begin{proof}
We again use the enumeration $D_{1},D_{2},\ldots$ of the finite subsets of $\mathbb{N}$ that contain $0$. The set of all $FU(m,D_{n})$ (where defined) is countable. Enumerate these spaces and denote this sequence $\{Y_{1},Y_{2},\ldots\}$. Let $r_i=i+\sum\limits_{j=1}^i diam(Y_j)$ for $i\ge 1$.

We then define $PU$ as the the total space of a metric resolution
$f:PU\to S$, where $S=\{r_i\}_{i\ge 1}$ is equipped with the max ultrametric
and $f^{-1}(r_i)$ is the corresponding $FU(m,D_{n})$ for each $i\ge 1$.

It is proper by \ref{SeparabilityPropernessOfUnions}. Let $(X,d)$ be a countable proper metric space of asymptotic dimension $0$. By Theorems \ref{countableintegral} and \ref{UltrametricAsCoarseDisjointUnion} we may assume without loss of generality that that $X$ can be written as a coarse disjoint union of finite $G_k$-ultrametric spaces $X_k$. There is a strictly increasing sequence $(n_{k})_{k\in\mathbb{N}}$ such that $G_k\subset D_{n_{k}}$ for each $k\ge 1$ and $D_{n_{k}}$ is a proper subset of $D_{n_{k+1}}$ for each $k\ge 1$.
Then, by Proposition \ref{TwoDifferentCoarseUnionsProp} we have that $X$ coarsely embeds into $PU$.
\end{proof}

The following result means that, in dimension $0$, our Theorem \ref{ProperDimZeroCase} is not implied by the work of Bell and Nag\'orko \cite{BN}:

\begin{Proposition}\label{ANonProperExample}
There is a separable space $X$ of asymptotic dimension $0$ that is not coarsely embeddable into a proper metric space.
\end{Proposition}
\begin{proof}
Consider a sequence of bounded $\mathbb{N}$-ultrametric spaces $X_i$ such that for each $n\ge 1$ there is infinitely many $i$ such that $X_i$ contains infinitely countably many points and the distance between any two different points is $n+1$. Let $X$ be a coarse disjoint union of the sequence $X_i$. If $X$ coarsely embeds into a proper metric space, then it contains $Y$ with the property that $B(Y,n)=X$ for some $n\ge 2$ and every bounded subset of $Y$ is finite. There is $i$ such that
$X_i$ contains infinitely countably many points, the distance between any two different points is $n+1$, and $dist(X_i,X\setminus X_i) > n$. 
Now, $Y\cap X_i$ must be finite, $B(Y,n)\cap X_i=B(Y\cap X_i,n)=Y\cap X_i$, a contradiction.
\end{proof}

\begin{Remark}
\ref{ANonProperExample} can be easily generalized to any dimension $n$.
\end{Remark}

\begin{Problem}
Given $n\ge 1$ is there a universal space in the class of proper metric spaces of asymptotic dimension at most $n$?
\end{Problem}

\section{Ultrametric groups as universal spaces}

In this section we show that certain unbounded ultrametric groups are universal in respective categories of spaces of asymptotic dimension $0$.

\begin{Definition}
An \textbf{(integral) ultrametric group} is a group equipped with a left-invariant (integral) ultrametric $d$. 
\end{Definition}

\begin{Proposition}\label{UltrametricsInducedBySubgroups}
Suppose $G$ is a group and $D$ is a discrete subset of non-negative reals containing $0$. Assigning $G$ a left-invariant $D$-ultrametric $d$ is equivalent to picking subgroups $G_a$, $a\in D$, of $G$ satisfying the following conditions:\\
1. $G_0=\{1_G\}$,\\
2. $G_a$ is a subgroup of $G_b$ if $a < b$ belong to $D$,\\
3. $\bigcup\limits_{a\in D}G_a=G$.
\end{Proposition}
\begin{proof}
Given a left-invariant $D$-ultrametric $d$ on $G$ and given $a\in D$ define $G_a$ as all $g\in G$ satisfying $d(g,1_G)\leq a$. Notice $g\in G_a$ implies $g^{-1}\in G_a$ as $d(g^{-1},1_G)=d(g\cdot g^{-1},g\cdot 1_G)=d(1_G,g)$.
Also, if $g,h\in G_a$, then $d(g\cdot h,1_G)\leq \max(d(g\cdot h,g),d(g,1_G))=\max(d(h,1_G),d(g,1_G))\leq a$. It is obvious that $\{G_a\}_{a\in D}$ satisfy Conditions 1-3.

Given $\{G_a\}_{a\in D}$ satisfying Conditions 1-3 define $d(g,h)$
as the infimum of $a\in D$ satisfying $g^{-1}\cdot h\in G_a$. If $d(g,h),d(h,k)\leq a$, then $g^{-1}\cdot h\in G_a$ and $ h^{-1}\cdot k\in G_a$,
so their product $g^{-1}\cdot k$ belongs to $G_a$ and $d(g,k)\leq a$. That means $d$ is an ultrametric, indeed.
\end{proof}

\begin{Definition}
Given a discrete subset $D$ of non-negative reals containing $0$ and given subgroups $G_a$, $a\in D$, of $G$ satisfying the following conditions:\\
1. $G_0=\{1_G\}$,\\
2. $G_a$ is a subgroup of $G_b$ if $a < b$ belong to $D$,\\
3. $\bigcup\limits_{a\in D}G_a=G$,\\
the ultrametric $d$ in \ref{UltrametricsInducedBySubgroups} is said to be \textbf{induced} by $\{G_a\}_{a\in D}$.
\end{Definition}

\begin{Proposition}\label{CoarseEquivalenceUltrametricStructures}
Suppose $G$ is a group and $d_i$, $i=1,2$, are two ultrametrics induced by families $\{G^i_a\}_{a\in D_i}$ of subgroups of $G$.
$(G,d_1)$ is coarsely equivalent to $(G,d_2)$ if and only if for each $a\in D_i$ and $j\ne i$ there is $b\in D_j$, such that $G^i_a\subset G^j_b$.
\end{Proposition}
\begin{proof}
Assume the identity $(G,d_1)\to (G,d_2)$ is large scale continuous (aka bornologous)
and $a\in D_1$. There is $b\in D_2$ such that $d_1(g,h)\leq a$ implies $d_2(g,h)\leq b$, so $g\in G^1_a$ implies $g\in G^1_b$
as $d_1(g,1_G)\leq a$ implies $d_2(g,1_G)\leq b$ and $g\in G^2_b$.

The reverse implication is similar.
\end{proof}

\begin{Proposition}\label{CoarseEmbeddingsViaBoundedSubsets}
Suppose $(X,d_X)$ is an integral ultrametric space and $(G,d_G)$ is an integral ultrametric group. If every bounded subset $B$ of $X$ isometrically embeds in $G$, then $(X,d_X)$ coarsely embeds in $(G,d_G)$.
\end{Proposition}
\begin{proof}
Of interest is only the case of $X$ being unbounded, so $G$ is also unbounded. Pick $x_0\in X$ and a sequence $\{x_n\}_{n\ge 1}$ of points in $X$ such that $d(x_{n+1},x_0) > d(x_n,x_0)+1$ for each $n\ge 1$.
Put $r_n=d(x_n,x_0)$ for $n\ge 1$ and pick an isometric embedding
$i_n:B_n\to G$, where $B_n=\{x\in X | r_{n-1} < d(x,x_0) \leq r_n$ for $n\ge 2$ and $B_1=\{x\in X | d(x,x_0) \leq r_1$. We may assume $i_n(x_n)=1_G$ for each $n\ge 1$.
Now pick a sequence $\{g_n\}_{n\ge 1}$ of elements of $G$ such that $d_G(g_1,1_G) > r_1$
and $s_n:=d_G(g_n,1_G) > d(g_{n-1},1_G)+diam(B_n)$. Replacing $i_n$
by $j_n:=g_n\cdot i_n$ we obtain a sequence of isometric embeddings
of $B_n$ into $C_n:= \{g\in G | s_{n-1} < d(g,1_G) \leq s_n\}$.
By \ref{IsometricembeddingsAndResolutions}, $X$ coarsely embeds in $G$.
\end{proof}

\begin{Corollary}\label{BasicCoarseEmbeddingIntoAGroup}
Suppose $(X,d_X)$ is an integral ultrametric space such that for each $n$ 
there is a cardinal number $c(n)$ with the property that each ball $B(x,n+2)$, $x\in X$, has cardinality at most $c(n)$. If $(G,d_G)$ is an integral ultrametric group induced by a sequence of subgroups $\{G_n\}_{n\ge 1}$ with the property that the cardinality of cosets of $G_n$ in $G_{n+1}$ is at least $c(n)$ for each $n\ge 1$, then $(X,d_X)$ coarsely embeds in $(G,d_G)$.
\end{Corollary}
\begin{proof}
Suppose each bounded subset of $X$ of diameter at most $n$ isometrically embeds in $G$. Therefore, for each $x\in X$, there is an isometric embedding
$i_x:B(x,n+1)\to G_n$ such that $i_x(x)=1_G$. Suppose $x_0\in X$. Consider the equivalence relation $x\sim y$ on $B(x_0,n+2)$ defined by $d_X(x,y) < n+1$. For each equivalence class $c$ not containing $x_0$ choose $x(c)\in B(x_0,n+2)\setminus B(x_0,n+1)$ and $g_c\in G_{n+1}$ such that if $c\ne k$, then $g_c^{-1}\cdot g_k\notin G_n$. 
Extend $i_{x_0}$ over $B(x_0,n+2)$ to a function $j$ by sending $x(c)$ to $g_c$ and by sending
any $x$ equivalent to $x_c$ to $g_c\cdot i_{x(c)}(x)$.
Notice $j$ is an isometric embedding when restricted to each equivalence class, the images of different equivalence classes are disjoint, and
if $d_X(x,y)=n+1$, then $d_G(j(x),j(y))=n+1$. That means $j$ is an isometric embedding.
\end{proof}

\begin{Corollary}
\label{UniversalGroupForSeparableCase}
Suppose $G$ is a countable group that is the union of an increasing sequence of its subgroups $\{G_i\}_{i\ge 1}^\infty$ with the property that the index of $G_i$ in $G_{i+1}$ is infinite for each $i\ge 1$. There is an integral ultrametric $d_G$ on $G$ such that $(G,d_G)$ is a universal space in the category of separable metric spaces of asymptotic dimension $0$.
\end{Corollary}

\begin{Corollary}
Let $G$ be a countable vector space over the rationals $Q$ that is of infinite algebraic dimension. There is an integral ultrametric $d_G$ on $G$ such that $(G,d_G)$ is a universal space in the category of separable metric spaces of asymptotic dimension $0$.
\end{Corollary}

\begin{Theorem}
Suppose $G$ is a countable group that is the union of a strictly increasing sequence of its finite subgroups $\{G_i\}_{i\ge 1}^\infty$.There is a proper integral ultrametric $d_G$ on $G$ such that $(G,d_G)$ is a universal space in the category of metric spaces of bounded geometry that have asymptotic dimension $0$.
\end{Theorem}
\begin{proof}
Consider a proper integral ultrametric space $(X,d_X)$ of bounded geometry and choose natural numbers $c(n)$ with the property that each ball $B(x,n+2)$, $x\in X$, contains at most $c(n)$ elements.
Replace $\{G_n\}$ by its subsequence $\{H_n\}$ such that the index of $H_n$ in $H_{n+1}$ is larger than $c(n+1)$
for each $n\ge 1$. 
By \ref{BasicCoarseEmbeddingIntoAGroup} and \ref{CoarseEquivalenceUltrametricStructures}, $(X,d_X)$ coarsely embeds into $(G,d_G)$.
\end{proof}

\begin{Corollary}
Let $G$ be a countable vector space over the $\mathbb{Z}/2\mathbb{Z}$ that is of infinite algebraic dimension. There is a proper integral ultrametric $d_G$ on $G$ such that $(G,d_G)$ is a universal space in the category of metric spaces of bounded geometry that have asymptotic dimension $0$.
\end{Corollary}

\begin{Remark}
See \cite{BZ} and \cite{BHZ} for coarse classifications of groups of asymptotic dimension $0$.
\end{Remark}


\begin{thebibliography}{99}
\bibitem{BHZ} T. Banakh, J. Higes, I. Zarichinyy, \emph{The coarse classification of countable abelian groups}, Transactions of the
American Mathematical Society,
Volume 362, Number 9, September 2010, Pages 4755--4780

\bibitem{BZ} T. Banakh and I. Zarichnyy, \emph{The coarse classification of homogeneous ultra-metric spaces}, arXiv:0801.2132

\bibitem{BN} G. C. Bell, A.Nag\'orko,
\emph{A new construction of universal spaces for asymptotic dimension},
Topology and its Applications Volume 160, Issue 1, Pages 159-169 (2013)


\bibitem{BH} M. Bridson and A. Haefliger, \emph{Metric spaces of non-positive curvature}, Grundlehren der mathematischen Wissenschaften (GL, volume 319), Springer- Verlag, Berlin, 1999.

\bibitem{BDHM}  N.Brodskiy, J.Dydak, J.Higes, and A.Mitra, \emph{   Dimension zero at all scales},
Topology and its Applications, 154 (2007), 2729--2740.

\bibitem{DZ}
A. Dranishnikov, and M. Zarichnyi, \emph{Universal spaces for asymptotic dimension}, Topology and its Applications, vol. 140, 203-225, 2004.

\bibitem{DK} C. Drutu, M. Kapovich, \emph{Geometric group theory}, Colloquium publications, Vol. 63, American Mathematics Society (2018).

\bibitem{DH}   J.Dydak and C.Hoffland, \emph{An alternative definition of coarse structures},
Topology and its Applications 155 (2008) 1013--1021


\bibitem{Engel}
R. Engelking, \emph{Theory of dimensions finite and infinite}, Sigma Series in Pure Mathematics,
vol. 10, Heldermann Verlag, 1995.

\bibitem{Grom}
M. Gromov, \emph{Asymptotic invariants for infinite groups}, in
Geometric Group Theory, vol. 2, 1--295, G. Niblo and M. Roller,
eds., Cambridge University Press, 1993.

\bibitem{Roe lectures}
J. Roe, \emph{Lectures on coarse geometry}, University Lecture
Series, 31. American Mathematical Society, Providence, RI, 2003.

\bibitem{Yu}
G. Yu, \emph{The coarse {B}aum-{C}onnes conjecture for spaces which admit a uniform embedding into {H}ilbert space}, Inventiones Mathematicae, vol. 139, 201-240, 2000.




\end{thebibliography}
\end{document}